\documentclass[a4paper,12pt]{article}
\usepackage[british]{babel}

\usepackage{hyperref}

\usepackage{color}

\usepackage{ucs}                
\usepackage[utf8x]{inputenc} 	

\usepackage[T1]{fontenc}     
\usepackage{lmodern}         

\usepackage[singlespacing]{setspace}

\usepackage[text={160mm,250mm},centering]{geometry}

\usepackage{mathrsfs}
\usepackage[sfdefault=cmbr,OMLmathsans]{isomath}  

\usepackage{authblk}
\usepackage{graphicx}
\usepackage{caption}

\usepackage{latexsym}
\usepackage{amsfonts,amssymb}
\usepackage{amsmath}
\usepackage{amsthm}

\usepackage{upgreek}  

\usepackage{tensor}  
\usepackage{epstopdf}  
\newcommand{\refS}[1]{Section~\ref{S:#1}}
\newcommand{\refSS}[1]{Subsection~\ref{SS:#1}}

\newcommand{\refeq}[1]{Eq.~(\ref{#1})}
\newcommand{\refig}[1]{Fig.~\ref{F:#1}}

\providecommand{\tbf}{\textbf}     

\providecommand{\mrm}{\mathrm}     

\providecommand{\C}{\mathcal}
\providecommand{\D}{\mathbb}
\providecommand{\E}{\mathscr}
\providecommand{\F}{\mathfrak}

\newtheorem{thm}{Theorem}

\newcommand{\ignore}[1]{}

\newcommand{\vphi}{\varphi}
\newcommand{\vpi}{\varpi}

\newcommand{\vek}[1]{\mathchoice{\displaystyle\boldsymbol{#1}}
{\textstyle\boldsymbol{#1}}{\scriptstyle\boldsymbol{#1}}
{\scriptscriptstyle\boldsymbol{#1}}}

\newcommand{\tnb}[1]{\mathchoice{\displaystyle\mathboldsans{#1}}
{\textstyle\mathboldsans{#1}}{\scriptstyle\mathboldsans{#1}}
{\scriptscriptstyle\mathboldsans{#1}}}

\newcommand{\vtil}[1]{\vek{\tilde{#1}}}

\newcommand{\EXP}[1]{\mathbb{E}\left(#1\right)}
 
\newcommand{\dlangle}{\langle\negthinspace\langle}
\newcommand{\drangle}{\rangle\negthinspace\rangle}

\newcommand{\tr}{\mathop{\mathrm{tr}}\nolimits}

\newcommand{\spn}{\mathop{\mathrm{span}}\nolimits}

\newcommand{\dd}{\partial}
\newcommand{\di}{\mathrm{d}}

\newcommand{\ii}{\mathchoice{\displaystyle\mathrm i}
{\textstyle\mathrm i}{\scriptstyle\mathrm i}
{\scriptscriptstyle\mathrm i}}

\newcommand{\citep}[1]{\cite{#1}}

\newcommand{\ip}[2]{\langle #1, #2 \rangle}

\newcommand{\ipd}[2]{\dlangle #1, #2 \drangle}

\newcommand{\nd}[1]{\| #1 \|}
\newcommand{\nt}[1]{\|\negthinspace| #1 |\negthinspace\|}

\newcommand{\var}{\mrm{var}}
\newcommand{\cov}{\mrm{cov}}

\definecolor{myred}{rgb}{1, 0.2, 0.2}

\newcommand{\authorhgm}{Hermann G. Matthies}
\newcommand{\authoral}{Alexander Litvinenko}
\newcommand{\authorbr}{Bojana V. Rosi\'c}
\newcommand{\authorez}{Elmar Zander}

\newcommand{\affilwire}{Institute of Scientific Computing \authorcr
                        Technische Universit\"at Braunschweig, Germany}
\newcommand{\affilkaust}{KAUST, Thuwal, Saudi Arabia}

\newcommand{\thetitle}{Parameter Estimation via Conditional Expectation --- A Bayesian Inversion}

\newcommand{\textdate}{\small Dedicated to Pierre Ladev\`eze on the occasion of his 70th birthday.}

\title{\thetitle\thanks{Partly supported by the Deutsche
          Forschungsgemeinschaft (DFG) through SFB 880.}}

\author[a]{\authorhgm\thanks{corresponding author}}
\author[a]{\authorez}
\author[a]{\authorbr}
\author[b]{\authoral}

\affil[a]{\affilwire}
\affil[b]{\affilkaust}

\date{\textdate}

\begin{document}

\maketitle

\begin{abstract}
When a mathematical or computational model is used to analyse some
system, it is usual that some parameters resp.\ functions or fields
in the model are not known, and hence uncertain.  These parametric
quantities are then identified by actual observations of the
response of the real system.  In a probabilistic setting, Bayes's
theory is the proper mathematical background for this identification
process.  The possibility of
being able to compute a conditional expectation turns out to be
crucial for this purpose.  We show how this theoretical background
can be used in an actual numerical procedure, and shortly discuss
various numerical approximations.
\end{abstract}

\section{Introduction} \label{S:intro}
The fitting of parameters resp.\ functions or fields --- these will all be
for the sake of brevity be referred to as parameters –-- in a
mathematical computational model is usually denoted as an inverse problem,
in contrast to predicting the output or state resp.\ response of the system
given certain inputs, which is called the forward problem.  In the
inverse problem, the response of the model is compared to the
response of the system.  The system may be a real world system, or
just another computational model --- usually a more complex one.
One then tries in various ways to match the model response with
the system response.

Typical deterministic procedures include such methods as
minimising the mean square error (MMSE), leading to optimisation
problems in the search of optimal parameters.  As the inverse
problem is typically ill-posed --- the observations do not do
contain enough information to uniquely determine the parameters ---
some additional oinformation has to be added to select a unique solution.
In the deterministic setting on then typically invokes additional
ad-hoc procedures like Tikhonov-regularisation
 \citep{Tikhonov1977, Tikhonov1995, Engl1987, Engl2000}.

In a probabilistic setting (e.g.\ \citep{jaynes03, Tarantola2004}
and references therein) the ill-posed problem becomes well-posed (e.g.\
\citep{Stuart2010}).  This is achieved at a cost though.  The unknown
parameters are considered as uncertain, and modelled as random
variables (RVs).  The information added is hence the \emph{prior}
probability distribution.  This means on one hand that the result of the
identification is a probability distribution, and not a single
value, and on the other hand the computational work may be
increased substantially, as one has to deal with RVs.  That the
result is a probability distribution may be seen as additional
information though, as it offers an assessment of the residual
uncertainty after the identification procedure, something which
is not readily available in the deterministic setting.  The
probabilistic setting thus can be seen as modelling our knowledge
about a certain situation --- the value of the parameters --- in the
language of probability theory, and using the observation to
update our knowledge, (i.e.\ the probabilistic description) by
\emph{conditioning} on the observation.

The key probabilistic background for this is Bayes's theorem in
the formulation of Laplace  \citep{jaynes03, Tarantola2004}.  It is well known
that the Bayesian update is theoretically based on the notion of
conditional expectation (CE) \citep{Bobrowski2006/087}.  Here we take an
approach which takes CE not only as a theoretical basis, but
also as a basic computational tool.  This may be seen as somewhat
related to the ``Bayes linear'' approach \citep{Goldstein2007, Kennedy01},
which has a linear approximation of CE as its basis, as will
be explained later.

In many cases, for example when tracking a dynamical system, the
updates are performed sequentially step-by-step, and for the next
step one needs not only a probability distribution in order to
perform the next step, but a random variable which may be evolved
through the state equation.  Methods on how to transform the prior RV
into the one which is conditioned on the observation will be
discussed as well \citep{HgmEzBvrAlOp15}.
This approach is very different to the very frequently
used one which refers to Bayes's theorem in terms of densities and
likelihood functions, and typically employs Markov-chain Monte
Carlo (MCMC) methods to sample from the posterior
(see e.g.\ \citep{Hastings1970, Marzouk2007, BvrAkJsOpHgm11}).

\section{Mathematical set-up}   \label{S:math-set-up}
Let us start with an example to have a concrete idea of what the
whole procedure is about.  Imagine a system described by a
diffusion equation, e.g.\ the diffusion of heat through a solid
medium, or even the seepage of groundwater through porous rocks
and soil: 
\begin{align} \label{eq:diff-eq} 
  \frac{\dd \tilde{\upsilon}}{\dd t}(x,t) &= \dot{\tilde{\upsilon}}(x,t) = \nabla\cdot
         (\kappa(x,\tilde{\upsilon}) \nabla \tilde{\upsilon}(x,t)) + \eta(x,t),\\
         \label{eq:diff-inibc} \tilde{\upsilon}(x,0) &= \tilde{\upsilon}_0(x) \quad \text{plus b.c.}
\end{align}
Here $x\in\C{G}$ is a spatial coordinate in the domain $\C{G} \subset
\D{R}^n$, $t \in [0,T]$ is the time, $\tilde{\upsilon}$ a scalar
function describing the diffusing quantity, $\kappa$ the (possibly
non-linear) diffusion tensor, $\eta$ external sources or
sinks, and $\nabla$ the Nabla operator.  Additionally assume
appropriate boundary conditions so that \refeq{eq:diff-eq} is well-posed.
Now, as often in such situations, imagine that we do not
know the initial conditions $\tilde{\upsilon}_0$ in \refeq{eq:diff-inibc} precisely, nor
the diffusion tensor $\kappa$, and maybe not even the driving source $\eta$,
i.e.\ there is some uncertainty attached as to what their precise values are.

\subsection{Data model}    \label{SS:data-m}
A more abstract setting which subsumes \refeq{eq:diff-eq} is to
view $\tilde{\upsilon}(t) := \tilde{\upsilon}(\cdot,t)$ as an element of a Hilbert-space
(for the sake of simplicity) $\C{V}$.  In the particular case of \refeq{eq:diff-eq}
one could take $\C{V}=\mrm{H}^1_E(\C{G})$, a closed subspace of the
Sobolev space $\mrm{H}^1(\C{G})$ incorporating the essential boundary
conditions.  Hence we may view \refeq{eq:diff-eq} and \refeq{eq:diff-inibc}
as an example of
\begin{equation} \label{eq:abstr-sys-I} 
 \frac{\di \tilde{\upsilon}}{\di t}(t) = \dot{\tilde{\upsilon}}(t) = A_{\C{V}}(q;\tilde{\upsilon}(t)) + \eta(q;t), 
   \quad \tilde{\upsilon}(0) = \tilde{\upsilon}_0(q) \in \C{V}, \; t\in[0,T].
\end{equation}
Here $A_{\C{V}}:\C{Q}\times\C{V}\to\C{V}$ is a possibly non-linear
operator in $\tilde{\upsilon}\in\C{V}$, and $q\in\C{Q}$ are the
parameters (like $\kappa$, $\tilde{\upsilon}_0$, or $\eta$, which more accurately
would be described as functions of $q$), where we assume for simplicity again
that $\C{Q}$ is some Hilbert space.
Both $A_{\C{V}}$, $\tilde{\upsilon}_0$, and $\eta$ could involve some noise, so that
one may view \refeq{eq:abstr-sys-I} as an instance of a stochastic
evolution equation.  This is our model of the system generating the
observed data, which we assume to be well-posed.

Hence assume further that we may observe a function $\hat{Y}(q;\tilde{\upsilon}(t))$
of the state $\tilde{\upsilon}(t)$
and the parameters $q$, i.e.\ $\hat{Y}:\C{Q}\times\C{V}\to\C{Y}$,
where we assume that $\C{Y}$ is a Hilbert space.
To make things simple, assume additionally that we observe $\hat{Y}(q;\tilde{\upsilon}(t))$
at regular time intervals $t_n = n \cdot \mrm{\Delta} t$,
i.e.\ we observe $y_{n}=\hat{Y}(q;\tilde{\upsilon}_n)$, where $\tilde{\upsilon}_n := \tilde{\upsilon}(t_n)$.
Denote the solution operator $\Upsilon$ of \refeq{eq:abstr-sys-I} as
\begin{equation} \label{eq:sol-op-I}  
 \tilde{\upsilon}_{n+1}= \Upsilon(t_{n+1},q,\tilde{\upsilon}_n,t_n,\eta),
\end{equation}
advancing the solution from $t_n$ to $t_{n+1}$.  Hence we are observing
\begin{equation}   \label{eq:obs-I}
 \hat{y}_{n+1} = \hat{h}(\hat{Y}(q;\Upsilon(t_{n+1},q,\tilde{\upsilon}_n,t_n,\eta)),v_n),
\end{equation}
where some noise $v_n$ --- inaccuracy of the observation --- has been
included, and $\hat{h}$ is an appropriate observation operator.  A simple
example is the often assumed additive noise
\[ \hat{h}(y,v) := y + S_{\C{V}}(\tilde{\upsilon})v,\]
where $v$ is a random vector, and for each $\tilde{\upsilon}$, $S_{\C{V}}(\tilde{\upsilon})$
is a bounded linear map to $\C{Y}$.

\subsection{Identification model}      \label{SS:ident-m}
Now that the model generating the data has been described, it
is the appropriate point to introduce the identification
model.  Similarly as before in \refeq{eq:abstr-sys-I}, we have a model
\begin{equation} \label{eq:abstr-sys-II} 
 \frac{\di u}{\di t}(t) = \dot{u}(t) = A(q;u(t)) + \eta(q;t), 
   \quad u(0) = u_0(q) \in \C{U}, \; t\in[0,T],
\end{equation}
which depends on the same parameters $q$ as in \refeq{eq:abstr-sys-I},
to be used for the identification, which we shall only write in its
abstract from.  Hence we assume that we can actually integrate
\refeq{eq:abstr-sys-II} from $t_n$ to $t_{n+1}$ with its solution operator $U$
\begin{equation} \label{eq:sol-op-II} 
 u_{n+1}= U(t_{n+1},q,u_n,t_n,\eta).
\end{equation}
Observe that the two spaces $\C{V}$ in \refeq{eq:abstr-sys-I} and $\C{U}$
in \refeq{eq:abstr-sys-II} are not the same, as in general we do not know
$\tilde{\upsilon} \in \C{V}$, we only have observations $y\in\C{Y}$.

As later not only the state $u\in\C{U}$ in \refeq{eq:abstr-sys-II}
has to be identified, but also the parameters $q$, and the identification
may happen sequentially, i.e.\ our estimate of $q$ will change from step $n$
to step $n+1$, we shall introduce an ``extended'' state vector 
$x=(u,q)\in\C{X}:=\C{Q}\times\C{U}$ and describe the change from
$n$ to $n+1$ by
\begin{equation} \label{eq:abs-II-ow} 
  x_{n+1} = (u_{n+1},q_{n+1}) = \hat{f}(x_n)
           := (U(t_{n+1},q_n,u_n,t_n,\eta),q_n).
\end{equation}
Let us explicitly introduce a noise $w\in\C{W}$ to cover the
stochastic contribution or modelling errors between \refeq{eq:abstr-sys-II}
and \refeq{eq:abstr-sys-I}, so that we set
\begin{equation} \label{eq:abs-II} 
  x_{n+1} = f(x_n,w_n),
\end{equation}
for example 
\[ f(x,w) = \hat{f}(x) + S_{\C{W}}(x) w, \]
where $w\in\C{W}$ is the random vector, and $S_{\C{W}}(x)\in\E{L}(\C{W},\C{X})$ is
for each $x \in \C{X}$ a bounded linear map from $\C{W}$ to $\C{X}$.

To deal with the extended state, we shall define the identification or
predicted observation operator as
\begin{equation}   \label{eq:obs-II}
 y_{n+1} = h(x_n,v_n) = H(x_{n+1},v_n) = H(f(x_n,w_n),v_n),
\end{equation}
where the noise $v_n$ --- the same as in \refeq{eq:obs-I},
our model of the inaccuracy of the observation --- has been
included.  A simple example with additive noise is
\[ h(x_n,v_n) := Y(q;U(t_{n+1},q_n,u_n,t_n,\eta)) + S_{\C{V}}(x_n)v_n,\]
where $v\in\C{V}$ is the random vector, and $S_{\C{V}}(x)\in\E{L}(\C{V},\C{X})$ is
for each $x \in \C{X}$ a bounded linear map from $\C{V}$ to $\C{X}$.
The mapping $Y:\C{Q}\times\C{U}\to\C{Y}$ is similar to the map
$\hat{Y}:\C{Q}\times\C{V}\to\C{Y}$ in the previous \refSS{data-m},
it predicts the ``true'' observation without noise $v_n$.
\refeq{eq:abs-II} for the time evolution of the extended state and
\refeq{eq:obs-II} for the observation are the basic building blocks for
the identification.

\section{Synopsis of Bayesian estimation}      \label{S:synopsis-Bayes}
There are many accounts of this, and this synopsis is just for the convenience
of the reader and to introduce notation.  Otherwise we refer to e.g.\
\citep{jaynes03, Tarantola2004, Goldstein2007, Kennedy01}, and in particular
for the r\^ole of conditional expectation (CE) to our work
\citep{BvrAkJsOpHgm11, HgmEzBvrAlOp15}.

The idea is that the observation $\hat{y}$ from \refeq{eq:obs-I}
depends on the unknown parameters $q$, which ideally should equal
$y_{n}$ from \refeq{eq:obs-II}, which in turn both directly and through the
state $x = (u(q),q)$ in \refeq{eq:abs-II} depends also on the
parameters $q$, should be equal, and any difference should give an indication
on what the ``true'' value of $q$ should be.  The problem in general
is --- apart from the distracting errors $w$ and $v$ --- that the mapping
$q \mapsto y=Y(q;u(q))$ is in general not invertible, i.e.\  $y$ does not contain
information to uniquely determine $q$, or there are many $q$ which give a good
fit for $\hat{y}$.  Therefore the \emph{inverse} problem of determining $q$
from observing $\hat{y}$ is termed an \emph{ill-posed} problem.

The situation is a bit comparable to Plato's allegory of the cave, where
Socrates compares the process of gaining knowledge with looking at the
shadows of the real things.  The observations $\hat{y}$ are the
``shadows'' of the ``real'' things $q$ and $\tilde{\upsilon}$, and from
observing the ``shadows'' $\hat{y}$ we want to infer what ``reality'' is,
in a way turning our heads towards it.  We hence
want to ``free'' ourselves from just observing the ``shadows'' and
gain some understanding of ``reality''.

One way to deal with this difficulty is to measure the difference
between observed $\hat{y}_n$ and predicted system output $y_n$ and
try to find parameters $q_n$ such that this difference is minimised.
Frequently it may happen that the parameters which realise the minimum
are not unique.  In case one wants a unique parameter, a choice has to
be made, usually by demanding additionally that some norm or similar
functional of the parameters is small as well, i.e.\ some regularity
is enforced. This optimisation approach hence leads to regularisation
procedures \citep{Tikhonov1977, Tikhonov1995, Engl1987, Engl2000}.

Here we take the view that our lack of knowledge or uncertainty
of the actual value of
the parameters can be described in a \emph{Bayesian} way through a
probabilistic model \citep{jaynes03, Tarantola2004}.  
The unknown parameter $q$ is then modelled as a random variable
(RV)---also called the \emph{prior} model---and additional information on the
system through measurement or observation
changes the probabilistic description to the so-called \emph{posterior} model.
The second approach is thus a method to update the probabilistic description 
in such a way as to take account of the additional information, and the 
updated probabilistic description \emph{is} the parameter estimate,
including a probabilistic description of the remaining uncertainty.

It is well-known that such a Bayesian update is in fact closely related
to \emph{conditional expectation}
\citep{jaynes03, Bobrowski2006/087, Goldstein2007, BvrAkJsOpHgm11, HgmEzBvrAlOp15},
and this will be the basis
of the method presented.  For these and other probabilistic notions
see for example \citep{Papoulis1998/107} and the references therein.
As the Bayesian update may be numerically
very demanding, we show computational procedures
to accelerate this update through methods based on 
\emph{functional approximation} or \emph{spectral representation}
of stochastic problems \citep{matthies6, HgmEzBvrAlOp15}.  These approximations
are in the simplest case known as Wiener's so-called \emph{homogeneous}
or \emph{polynomial chaos} expansion,
which are polynomials in independent Gaussian RVs---the ``chaos''---and which
can also be used numerically in a Galerkin procedure
\citep{matthies6, HgmEzBvrAlOp15}.

Although the Gauss-Markov theorem and its extensions \citep{Luenberger1969}
are well-known, as well as its connections to the Kalman filter
\citep{Kalman, Grewal2008}---see also the recent Monte Carlo or
\emph{ensemble} version \citep{Evensen2009}---the connection to Bayes's theorem is not
often appreciated, and is sketched here.   This turns out to be a linearised version
of \emph{conditional expectation}.

Since the parameters of the model to be estimated are uncertain, all relevant
information may be obtained via their stochastic description.
In order to extract information from the posterior, most estimates take
the form of expectations w.r.t.\ the posterior, i.e.\ a conditional
expectation (CE).  These expectations---mathematically integrals, numerically to
be evaluated by some quadrature rule---may be computed via asymptotic,
deterministic, or sampling methods by typically computing first the posterior
density.  As we will see, the posterior density does not always exist \citep{rao2005}.
Here we follow our recent publications \citep{opBvrAlHgm12, BvrAkJsOpHgm11, HgmEzBvrAlOp15}
and introduce a novel approach, namely computing the CE directly and not via
the posterior density \citep{HgmEzBvrAlOp15}.  This way all relevant information
from the conditioning may be computed directly.  In addition, we want to change the
prior, represented by a random variable (RV), into a new random variable which has
the correct posterior distribution.  We will discuss how this may be achieved, and
what approximations one may employ in the computation.

To be a bit more formal, assume that the uncertain parameters are given by
\begin{equation}  \label{eq:RVp}
x: \Omega \to  \C{X} \text{  as a RV on a probability space   }
  (\Omega, \F{A}, \D{P}) ,
\end{equation}
where the set of elementary events is $\Omega$, $\F{A}$ a $\sigma$-algebra of
measurable events, and $\D{P}$ a probability measure.  The \emph{expectation}
corresponding to $\D{P}$ will be denoted by $\EXP{}$, e.g.\ 
\[ \bar{\Psi}:=\EXP{\Psi} := \int_\Omega \Psi(x(\omega)) \, \D{P}(\di \omega), \]
for any measurable function $\Psi$ of $x$.

  Modelling our lack-of-knowledge
about $q$ in a Bayesian way \citep{jaynes03, Tarantola2004, Goldstein2007}
by replacing them with  random variables (RVs), the problem
becomes well-posed \citep{Stuart2010}.  But of course one is looking now at the
problem of finding a probability distribution that best fits the data;
and one also obtains a probability distribution, not just \emph{one} value
$q$.  Here we focus on the use of procedures similar to a linear Bayesian approach 
\citep{Goldstein2007} in the framework of ``white noise'' analysis.

As formally $q$ is a RV, so is the state $x_n$ of \refeq{eq:abs-II},
reflecting the uncertainty about the parameters and state of \refeq{eq:abstr-sys-I}.
From this follows that also the prediction of the measurement 
$y_n$ \refeq{eq:obs-II} is a RV; i.e.\ we have a \emph{probabilistic}
description of the measurement.

\subsection{The theorem of Bayes and Laplace}      \label{SS:Bayes-Laplace}
Bayes original statement of the theorem which today bears his name was
only for a very special case.  The form which we know today is due to
Laplace, and it is a statement about conditional probabilities.  A good
account of the history may be found in \citep{McGrayne2011}.

Bayes's theorem is commonly accepted as a consistent way to incorporate
new knowledge into a probabilistic description \citep{jaynes03, Tarantola2004}.
The elementary textbook statement of the theorem is about
conditional probabilities
\begin{equation}  \label{eq:iII}
 \D{P}(\C{I}_x|\C{M}_y) = \frac{\D{P}(\C{M}_y|\C{I}_x)}{\D{P}(\C{M}_y)}\D{P}(\C{I}_x),
 \quad \text{if }\D{P}(\C{M}_y)>0,
\end{equation}
where $\C{I}_x \subset \C{X}$ is some subset of possible $x$'s on which we would like
to gain some information, and $\C{M}_y\subset \C{Y}$ is the information
provided by the measurement.  The term $\D{P}(\C{I}_x)$ is the so-called
\emph{prior}, it is what we know before the observation $\C{M}_y$.
The quantity $\D{P}(\C{M}_y|\C{I}_x)$ is the so-called \emph{likelihood},
the conditional probability of $\C{M}_y$ assuming that $\C{I}_x$ is given.
The term $\D{P}(\C{M}_y)$ is the so called \emph{evidence}, the probability
of observing $\C{M}_y$ in the first place, which sometimes can be expanded
with the \emph{law of total probability}, allowing to choose between
different models of explanation.  It is necessary to make the
right hand side of \refeq{eq:iII} into a real probability---summing
to unity---and hence the term $\D{P}(\C{I}_x|\C{M}_y)$, the \emph{posterior}
reflects our knowledge on $\C{I}_x$ \emph{after} observing $\C{M}_y$.
The quotient $\D{P}(\C{M}_y|\C{I}_x)/\D{P}(\C{M}_y)$ is sometimes termed the
\emph{Bayes} factor, as it reflects the relative change in probability after
observing $\C{M}_y$.

This statement \refeq{eq:iII} runs into problems if the set observations
$\C{M}_y$ has vanishing measure, $\D{P}(\C{M}_y)=0$, as is the case when we observe
\emph{continuous} random variables, and the theorem would have to be
formulated in \emph{densities}, or more precisely in 
probability density functions (pdfs).  But the Bayes factor then has the indeterminate
form $0/0$, and some form of limiting procedure is needed.  As a sign that
this is not so simple---there are different and inequivalent forms of doing
it---one may just point to the so-called \emph{Borel-Kolmogorov} paradox.
See \citep{rao2005} for a thorough discussion.

There is one special case where something resembling \refeq{eq:iII} may be
achieved with pdfs, namely if $y$ and $x$ have a 
\emph{joint} pdf $\pi_{y,x}(y,x)$.
As $y$ is essentially a function of $x$, this is a special case
depending on conditions on the error term $v$.  In this case \refeq{eq:iII}
may be formulated as 
\begin{equation}  \label{eq:iIIa}
 \pi_{x|y}(x|y) = \frac{\pi_{y,x}(y,x)}{Z_s(y)},
\end{equation}
where $\pi_{x|y}(x|y)$ is the \emph{conditional} pdf, and the ``evidence'' $Z_s$
(from German \emph{Zustandssumme} (sum of states), a term used in physics)
is a normalising factor such that the conditional pdf $\pi_{x|y}(\cdot|y)$
integrates to unity
\[ Z_s(y) = \int_\Omega  \pi_{y,x}(y,x) \, \di x . \]
The joint pdf may be split into the \emph{likelihood density} $\pi_{y|x}(y|x)$
and the \emph{prior} pdf $\pi_x(x)$
\[ \pi_{y,x}(y,x) =  \pi_{y|x}(y|x) \pi_x(x) , \]
so that \refeq{eq:iIIa} has its familiar form (\citep{Tarantola2004} Ch.\ 1.5)
\begin{equation}  \label{eq:iIIa1}
 \pi_{x|y}(x|y) = \frac{\pi_{y|x}(y|x)}{Z_s(y)}  \pi_x(x) .
\end{equation}
These terms are in direct correspondence with
those in \refeq{eq:iII} and carry the same names.
Once one has the conditional measure $\D{P}(\cdot|\C{M}_y)$
or even a conditional pdf $\pi_{x|y}(\cdot|y)$,
the \emph{conditional expectation} (CE) $\EXP{\cdot|\C{M}_y}$ may be
defined as an integral over that
conditional measure resp.\ the conditional pdf.  Thus classically,
the conditional measure or pdf implies the conditional expectation:
\[ \EXP{\Psi|\C{M}_y} := \int_{\C{X}} \Psi(x) \, \D{P}(\di x|\C{M}_y) \]
for any measurable function $\Psi$ of $x$.

Please observe that
the model for the RV representing the error $v(\omega)$ determines
the likelihood functions $\D{P}(\C{M}_y|\C{I}_x)$ resp.\ the existence and form
of the likelihood density $\pi_{y|x}(\cdot|x)$.  
In computations, it is here that the computational
model \refeq{eq:abstr-sys-II} and \refeq{eq:obs-II} is needed,
to predict the measurement RV $y$ given
the state and parameters $x$ as a RV.

Most computational approaches determine the pdfs, but we will later argue that
it may be advantageous to work directly with RVs, and not with conditional
probabilities or pdfs.  To this end, the concept of conditional expectation (CE)
and its relation to Bayes's theorem is needed \citep{Bobrowski2006/087}.

\subsection{Conditional expectation}      \label{SS:condex}
To avoid the difficulties with conditional probabilities like in the
Borel-Kolmogorov paradox alluded to in the previous \refSS{Bayes-Laplace}, 
\emph{Kolmogorov} himself---when he was setting up
the axioms of the mathematical theory probability---turned the relation between
conditional probability or pdf and conditional expectation around, and
defined as a first and fundamental notion \emph{conditional expectation}
 \citep{Bobrowski2006/087, rao2005}.

It has to be defined not with respect to measure-zero observations of a RV $y$,
but w.r.t\ sub-$\sigma$-algebras $\F{B}\subset\F{A}$ of the underlying
$\sigma$-algebra $\F{A}$.  The $\sigma$-algebra may be loosely seen as the
collection of subsets of $\Omega$ on which we can make statements about their
probability, and for fundamental mathematical reasons in many cases this is
\emph{not} the set of \emph{all} subsets of $\Omega$.  The sub-$\sigma$-algebra $\F{B}$
may be seen as the sets on which we learn something through the observation.

The simplest---although slightly restricted---way to define the conditional
expectation \citep{Bobrowski2006/087} is to just consider RVs with 
\emph{finite variance}, i.e.\ the Hilbert-space 
\[ \C{S} := \mrm{L}_2(\Omega,\F{A},\D{P}) := \{r:\Omega\to\D{R}\;:\;
  r \;\text{measurable w.r.t.}\;\F{A}, \EXP{|r|^2}<\infty \}.\]
If $\F{B}\subset\F{A}$ is a sub-$\sigma$-algebra, the space
\[ \C{S}_\F{B} := \mrm{L}_2(\Omega,\F{B},\D{P}) := \{r:\Omega\to\D{R}\;:\;  
  r \;\text{measurable w.r.t.}\;\F{B}, \EXP{|r|^2}<\infty \} \subset \C{S}\]
is a \emph{closed} subspace, and hence has a well-defined continuous orthogonal
projection $P_\F{B}: \C{S}\to\C{S}_\F{B}$.  The \emph{conditional expectation} (CE)
of a RV $r\in\C{S}$
w.r.t.\ a sub-$\sigma$-algebra $\F{B}$ is then defined as that orthogonal projection
\begin{equation} \label{eq:def-ce}
\EXP{r|\F{B}} :=   P_\F{B}(r) \in \C{S}_\F{B}.
\end{equation}
It can be shown \citep{Bobrowski2006/087} to coincide with the classical notion
when that one is defined, and the \emph{unconditional} expectation $\EXP{}$
is in this view just the CE w.r.t.\ the minimal $\sigma$-algebra 
$\F{B}=\{\emptyset, \Omega\}$.
As the CE is an orthogonal projection, it minimises the squared error
\begin{equation} \label{eq:min-ce}
 \EXP{|r - \EXP{r|\F{B}}|^2} = \min\{ \EXP{|r - \tilde{r}|^2}\;:\; 
   \tilde{r}\in\C{S}_\F{B} \},
\end{equation}
from which one obtains the \emph{variational equation} or orthogonality relation
\begin{equation} \label{eq:var-ce}
\forall \tilde{r}\in\C{S}_\F{B}: \quad \EXP{\tilde{r} (r - \EXP{r|\F{B}})} =0 ;
\end{equation}
and one has a form of \emph{Pythagoras's} theorem 
\[ \EXP{|r|^2} = \EXP{|r - \EXP{r|\F{B}}|^2}  +  \EXP{|\EXP{r|\F{B}}|^2}. \]
The CE is therefore a form of a \emph{minimum mean square error} (MMSE) estimator.

Given the CE, one may completely characterise the \emph{conditional} probability,
e.g.\ for $A \subset \Omega, A \in \F{B}$ by 
\[ \D{P}(A | \F{B}) := \EXP{\chi_A|\F{B}} , \]
where $\chi_A$ is the RV which is unity iff $\omega\in A$ and vanishes otherwise
--- the \emph{usual} characteristic function, sometimes also termed an indicator
function.  Thus if we know $\D{P}(A | \F{B})$ for each $A \in \F{B}$, we know
the conditional probability.
Hence having the CE $\EXP{\cdot|\F{B}}$ allows one to know everything about
the conditional probability; the conditional or posterior density is not needed.
If the prior probability was the distribution of some
RV $r$, we know that it is completely characterised by the \emph{prior} characteristic
function --- in the sense of probability theory --- $\vphi_r(s) := \EXP{\exp(\ii r s)}$.
To get the \emph{conditional} characteristic function $\vphi_{r|\F{B}}(s) =
\EXP{\exp(\ii r s)|\F{B}}$, all one has to do is use the CE instead of the unconditional
expectation.  This then completely characterises the conditional distribution.

In our case of an observation of a RV $y$,
the sub-$\sigma$-algebra $\F{B}$ will be the one generated by
the \emph{observation} $y=h(x,v)$, i.e.\ $\F{B}=\sigma(y)$, these are those
subsets of $\Omega$ on which we may obtain \emph{information} from the
observation.  According to the \emph{Doob-Dynkin} lemma the subspace
$\C{S}_{\sigma(y)}$ is given by
\begin{equation} \label{eq:DDL}
\C{S}_{\sigma(y)} :=  \{r \in \C{S} \;:\; r(\omega) = \phi(y(\omega)),  
  \phi \;\text{measurable} \} \subset \C{S} ,
\end{equation}
i.e.\ functions of the observation.  This means intuitively that anything we
learn from an observation is a function of the observation, and the subspace
$\C{S}_{\sigma(y)} \subset \C{S} $ is where the information from the measurement
is lying.

Observe that the CE $\EXP{r|\sigma(y)}$ and conditional probability
$\D{P}(A|\sigma(y))$---which we will abbreviate to
$\EXP{r|y}$, and similarly $\D{P}(A | \sigma(y))=\D{P}(A|y)$---is a RV,
as $y$ is a RV.  Once an observation has been made,
i.e.\ we observe for the RV $y$ the fixed value $\hat{y}\in\C{Y}$,
then---for almost all $\hat{y}\in\C{Y}$---
$\EXP{r|\hat{y}} \in \D{R}$ is just a number---the \emph{posterior expectation},
and $\D{P}(A|\hat{y})=\EXP{\chi_A|\hat{y}}$ is the \emph{posterior probability}.
Often these are also termed conditional expectation and conditional probability,
which leads to confusion.  We therefore prefer the attribute \emph{posterior}
when the actual observation $\hat{y}$ has been observed and inserted in the expressions.
Additionally, from \refeq{eq:DDL} one knows that for
some function $\phi_r$ --- for each RV $r$ it is a possibly different function ---
one has that
\begin{equation}  \label{eq:CE-DDL}
  \EXP{r|y} = \phi_r(y) \quad \text{ and } \quad \EXP{r|\hat{y}} = \phi_r(\hat{y})
\end{equation}

In relation to Bayes's theorem, one may conclude that if it is possible to
compute the CE w.r.t.\ an observation $y$ or rather the posterior expectation,
then the conditional and especially the posterior probabilities after the
observation $\hat{y}$ may as well be computed, regardless whether joint
pdfs exist or not.  We take this as the starting point to Bayesian estimation.

The conditional expectation has been formulated for scalar RVs, but it is clear
that the notion carries through to vector-valued RVs in a straightforward manner,
formally by seeing a---let us say---$\C{Y}$-valued RV as an element of the
tensor Hilbert space $\E{Y}=\C{Y}\otimes\C{S}$ \citep{Hackbusch_tensor}, as
\[ \E{Y}=\C{Y}\otimes\C{S} \cong \mrm{L}_2(\Omega,\F{A},\D{P};\C{Y}), \]
the RVs in $\C{Y}$ with finite \emph{total} variance
\[ \nt{\tilde{y}}_{\E{Y}}^2 = \int_\Omega \nd{\tilde{y}(\omega)}_{\C{Y}}^2
        \,\D{P}(\di \omega) < \infty . \]
Here $\nd{\tilde{y}(\omega)}_{\C{Y}}^2 = \ip{\tilde{y}(\omega)}{\tilde{y}(\omega)}_{\C{Y}}$
is the norm squared on the deterministic component $\C{Y}$ with inner product
$\ip{\cdot}{\cdot}_{\C{Y}}$; and the total $\mrm{L}_2$-norm
of an elementary tensor $y\otimes r\in \C{Y}\otimes\C{S}$ with $y\in\C{Y}$
and $r\in\C{S}$ can also be written as
\[  \nt{y\otimes r}_{\E{Y}}^2 = \ipd{y\otimes r}{y\otimes r}_{\E{Y}}= 
    \nd{y}_{\C{Y}}^2 \nd{r}_{\C{S}}^2 = \ip{y}{y}_{\C{Y}} \ip{r}{r}_{\C{S}}, \]
where $\ip{r}{r}_{\C{S}} = \nd{r}_{\C{S}}^2 := \EXP{|r|^2}$ is the usual inner product
of scalar RVs.

The CE on $\E{Y}$ is then formally
given by $\D{E}_{\E{Y}}(\cdot|\F{B}):=I_{\C{Y}}\otimes\EXP{\cdot|\F{B}}$, where
$I_{\C{Y}}$ is the identity operator on $\C{Y}$.  This means that for an elementary
tensor $y\otimes r \in \C{Y}\otimes\C{S}$ one has
\[ \D{E}_{\E{Y}}(y\otimes r|\F{B}) = y \otimes \EXP{r|\F{B}}. \]
The vector valued conditional expectation 
\[ \D{E}_{\E{Y}}(\cdot|\F{B}) = I_{\C{Y}}\otimes\EXP{\cdot|\F{B}}: 
    \E{Y} = \C{Y}\otimes\C{S} \to \C{Y}\] 
is also an orthogonal projection, but
in $\E{Y}$, for simplicity also denoted by $\EXP{\cdot|\F{B}} = P_{\F{B}}$ when there
is no possibility of confusion.

\section{Constructing a posterior random variable}      \label{S:construct-post}
We recall the equations governing our model \refeq{eq:abs-II} and \refeq{eq:obs-II}, 
and interpret them now as equations acting on RVs, i.e.\ for $\omega \in \Omega$:
\begin{align}  \label{eq:abs-RV}
   \hat{x}_{n+1}(\omega) &= f(x_n(\omega),w_n(\omega)), \\
   \label{eq:obs-RV}
   y_{n+1}(\omega) &= h(x_n(\omega),v_n(\omega)),
\end{align}
where one may now see the mappings $f:\E{X}\times\E{W}\to\E{X}$ and
$h:\E{X}\times\E{V}\to\E{Y}$ acting on the tensor Hilbert spaces of RVs with
finite variance, e.g.\ $\E{Y} := \C{Y}\otimes\C{S}$ with the inner product
as explained in \refSS{condex}; and similarly for $\E{X} := \C{X}\otimes\C{S}$
resp.\ $\E{W}$ and $\E{V}$.

\subsection{Updating random variables}      \label{SS:updating-RV}
We now focus on the step from time $t_n$ to $t_{n+1}$.  Knowing the RV $x_n \in\E{X}$,
one predicts the new state $\hat{x}_{n+1}\in\E{X}$ and the measurement $y_{n+1}\in\E{Y}$.
With the CE operator from the measurement prediction $y_{n+1}$ in \refeq{eq:obs-RV}
\begin{equation}  \label{eq:CE-Psi}
   \EXP{\Psi(x_{n+1})|\sigma(y_{n+1})} = \phi_\Psi(y_{n+1}) ,
\end{equation}   
and the actual observation $\hat{y}_{n+1}$ one may then compute the \emph{posterior}
expectation operator
\begin{equation}  \label{eq:Post-Psi}
   \EXP{\Psi(x_{n+1})|\hat{y}_{n+1}} = \phi_\Psi(\hat{y}_{n+1}).
\end{equation}
This has all the information about the posterior probability.

But to then go on from $t_{n+1}$ to $t_{n+2}$ with the \refeq{eq:abs-RV} and \refeq{eq:obs-RV},
one needs a new RV $x_{n+2}$ which
has the posterior distribution described by the mappings $\phi_\Psi(\hat{y}_{n+1})$ in
\refeq{eq:Post-Psi}.  Bayes's theorem only specifies this probabilistic content.
There are many RVs which have this posterior distribution, and we have to pick
a particular representative to continue the computation.
We will show a method which in the simplest case comes back to MMSE.

Here it is proposed to construct this new RV $x_{n+1}$ from
the predicted $\hat{x}_{n+1}$ in \refeq{eq:abs-RV} with a mapping, starting
from very simple ones and getting ever more complex.
For the sake of brevity of notation, the forecast RV will be called $x_f = \hat{x}_{n+1}$,
and the forecast measurement $y_f = y_{n+1}$, and we will denote the measurement just
by $\hat{y}=\hat{y}_{n+1}$.  The RV $x_{n+1}$ we want to construct will
be called the \emph{assimilated} RV $x_a = x_{n+1}$ --- it has assimilated the new
observation $\hat{y}=\hat{y}_{n+1}$.
Hence what we want is a new RV which is an \emph{update} of the forecast RV $x_f$
\begin{equation}   \label{eq:update-all}
   x_a = B(x_f,y_f,\hat{y}) = x_f + \Xi(x_f,y_f,\hat{y}),
\end{equation}
with a Bayesian update map $B$ resp.\ a change given by the \emph{innovation} map $\Xi$.
Such a transformation is often called a \emph{filter} --- the measurement $\hat{y}$
is filtered to produce the update.

\subsection{Correcting the mean}      \label{SS:correct-mean}
We take first the task to give the new RV the correct posterior \emph{mean}
$\bar{x}_a = \EXP{x_a|\hat{y}}$, i.e.\ we take $\Psi(x)=x$ in
\refeq{eq:Post-Psi}.  Remember that according to \refeq{eq:def-ce}
$\EXP{x_a|\sigma(y_f)} = \phi_{x_f}(y_f) =: \phi_x(y_f)$
is an orthogonal projection $P_{\sigma(y_f)}(x_f)$
from $\E{X} = \C{X}\otimes \C{S}$ onto $\E{X}_\infty := \C{X}\otimes \C{S}_\infty$, where
$\C{S}_\infty := \C{S}_{\sigma(y)}=\mrm{L}_2(\Omega,\sigma(y_f),\D{P})$.
Hence there is an orthogonal decomposition
\begin{align}  \label{eq:ortho-space}
   \E{X} &= \C{X}\otimes \C{S} = \E{X}_\infty \oplus \E{X}_\infty^\perp
          = (\C{X}\otimes \C{S}_\infty) \oplus (\C{X}\otimes \C{S}_\infty^\perp), \\
    \label{eq:ortho-RV}
    x_f &=   P_{\sigma(y_f)}(x_f) + (I_{\E{X}} - P_{\sigma(y_f)})(x_f) =
           \phi_x(y_f) + (x_f - \phi_x(y_f)).
\end{align}
As $ P_{\sigma(y_f)} = \EXP{\cdot|\sigma(y_f)}$ is a projection, one sees from
\refeq{eq:ortho-RV} that the second term has vanishing CE for any measurement $\hat{y}$:
\begin{equation}  \label{eq:corr-mean}
\EXP{x_f - \phi_x(y_f)|\sigma(y_f)} = P_{\sigma(y_f)}(I_{\E{X}} - P_{\sigma(y_f)})(x_f) =0.
\end{equation}
One may view this also in the following way:  From the measurement $y_a$ resp.\  $\hat{y}$
we only learn something about the subspace $\E{X}_\infty$.  Hence when the measurement comes,
we change the decomposition \refeq{eq:ortho-RV} by only fixing the component $\phi_x(y_f)
\in \E{X}_\infty$, and leaving the orthogonal rest unchanged:
\begin{equation}   \label{eq:update-xa1}
  x_{a,1} = \phi_x(\hat{y}) + (x_f - \phi_x(y_f)) = x_f + (\phi_x(\hat{y}) - \phi_x(y_f)).
\end{equation}
Observe that this is just a linear \emph{translation} of the RV $x_f$, i.e.\ a very simple
map $B$ in \refeq{eq:update-all}.
From \refeq{eq:corr-mean} follows that
\[  \bar{x}_{a,1} = \EXP{x_{a,1}|\hat{y}} = \phi_x(\hat{y}) = \EXP{x_{a}|\hat{y}}, \]
hence the RV $x_{a,1}$ from \refeq{eq:update-xa1} has the \emph{correct} posterior mean.

Observe that according to \refeq{eq:corr-mean} the term $x_\perp :=(x_f - \phi_x(y_f))$ in
\refeq{eq:update-xa1} is a zero mean RV, hence the covariance and total variance of $x_{a,1}$
is given by
\begin{align}  \label{eq:cov-xa1}
   \cov(x_{a,1}) &= \EXP{x_\perp \otimes x_\perp} = \EXP{x_\perp^{\otimes 2}} =: C_1, \\
    \label{eq:toc-xa1}
    \var(x_{a,1}) &=  \EXP{\nd{x_\perp(\omega)}_{\C{X}}^2} = \tr(\cov(x_{a,1})).
\end{align}

\subsection{Correcting higher moments}      \label{SS:correct-higher}
Here let us just describe two small additional steps:  
we take $\Psi(x)=\nd{x-\phi_x(\hat{y})}_{\C{X}}^2$ in \refeq{eq:Post-Psi}, and hence obtain
the total posterior variance as
\begin{equation}  \label{eq:var-xa}
   \var(x_{a}) =  \EXP{\nd{x_f - \phi_x(y_f)}_{\C{X}}^2|\hat{y}} = \phi_{x-\bar{x}}(\hat{y}).
\end{equation}
Now it is easy to correct \refeq{eq:update-xa1} to obtain
\begin{equation}   \label{eq:update-xat}
  x_{a,t} = \phi_x(\hat{y}) + \sqrt{\frac{\var(x_{a})}{\var(x_{a,1})}}(x_f - \phi_x(y_f)),
\end{equation}
a RV which has the correct posterior mean \emph{and} the correct posterior total variance
\[ \var(x_{a,t}) = \var(x_{a}) . \]
Observe that this is just a linear translation and partial scaling of the RV $x_f$,
i.e.\ still a very simple map $B$ in \refeq{eq:update-all}.

With more computational effort, one may choose $\Psi(x)=(x-\phi_x(\hat{y}))^{\otimes 2}$ 
in \refeq{eq:Post-Psi}, to obtain the covariance of $x_a$:
\begin{equation}  \label{eq:cov-xa}
   \cov(x_{a}) =  \EXP{(x-\phi_x(\hat{y}))^{\otimes 2}|\hat{y}} = \phi_{\otimes 2}(\hat{y})  
      =: C_a .
\end{equation}
Instead of just scaling the RV as in \refeq{eq:update-xat}, one may now take
\begin{equation}   \label{eq:update-xa2}
  x_{a,2} = \phi_x(\hat{y}) + B_a {B_1}^{-1}(x_f - \phi_x(y_f)),
\end{equation}
where $B_1$ is any operator ``square root'' that satisfies $B_{1} {B_{1}}^{*} = C_1$ in
\refeq{eq:cov-xa1}, and similarly $B_{a}  {B_{a}}^{*} = C_a$ in \refeq{eq:cov-xa}.
One possibility is the real square root --- as $C_1$ and $C_a$ are positive definite ---
$B_1 = C_1^{1/2}$, but computationally a Cholesky factor is usually cheaper.
In any case, no matter which ``square root'' is chosen, the RV $x_{a,2}$ in
\refeq{eq:update-xa2} has the correct posterior mean \emph{and} the correct
posterior covariance.
Observe that this is just an affine transformation of the RV $x_f$, i.e.\ still
a fairly simple map $B$ in \refeq{eq:update-all}.

By combining further transport maps \citep{moselhyYMarz:2011} it seems
possible to construct a RV $x_a$ which has the desired posterior distribution
to any accuracy.  This is beyond the scope of the present paper, and is
ongoing work on how to do it in the simplest way.  For the following, we shall be
content with the update \refeq{eq:update-xa1} in \refSS{correct-mean}.

\section{The Gauss-Markov-Kalman filter (GMKF)}      \label{S:GMKF}
It turned out that practical computations in the context of Bayesian estimation
can be extremely demanding, see \citep{McGrayne2011} for an account of the
history of Bayesian theory, and the break-throughs required in computational
procedures to make Bayesian estimation possible at all for practical purposes.
This involves both the Monte Carlo (MC) method and the Markov chain Monte Carlo
(MCMC) sampling procedure.  One may have gleaned this also already from \refS{construct-post}.

To arrive at computationally feasible procedures for computationally
demanding models \refeq{eq:abs-RV} and \refeq{eq:obs-RV}, where MCMC methods are not feasible,
approximations are necessary.  This means in some way not using all information
but having a simpler computation.  Incidentally, this connects with the
Gauss-Markov theorem \citep{Luenberger1969} and the Kalman filter (KF)
\citep{Kalman, Grewal2008}.  These were initially stated and developed without
any reference to Bayes's theorem.  The Monte Carlo (MC) computational
implementation of this is the \emph{ensemble} KF (EnKF) \citep{Evensen2009}.
We will in contrast use a white noise or polynomial chaos approximation
\citep{opBvrAlHgm12, BvrAkJsOpHgm11, HgmEzBvrAlOp15}.  But the initial ideas leading to
the abstract Gauss-Markov-Kalman filter (GMKF) are independent of any
computational implementation and are presented first.  It is in an abstract
way just \emph{orthogonal projection}, based on the update \refeq{eq:update-xa1}
in \refSS{correct-mean}.

\subsection{Building the filter}     \label{SS:building-filter}
Recalling \refeq{eq:abs-RV} and \refeq{eq:obs-RV} together with \refeq{eq:update-xa1},
the algorithm for forecasting and assimilating with just the posterior mean
looks like
\begin{align*}   
   \hat{x}_{n+1}(\omega) &= f(x_n(\omega),w_n(\omega)), \\
   y_{n+1}(\omega) &= H(f(x_n(\omega),w_n(\omega)),v_n(\omega)), \\
  x_{n+1}(\omega) &=  \hat{x}_{n+1}(\omega) + (\phi_x(\hat{y}_{n+1}) - \phi_x(y_{n+1}(\omega))).
\end{align*}
For simplicity of notation the argument $\omega$ will be suppressed.  Also it will
turn out that the mapping $\phi_x$ representing the CE can in most cases only be
computed approximately, so we want to look at update algorithms with a general map
$g:\C{Y} \to \C{X}$ to possibly approximate $\phi_x$:
\begin{multline}      \label{eq:update-xa1-filt}
  x_{n+1} =  f(x_n,w_n) + (g(\hat{y}_{n+1}) - g(H(f(x_n,w_n),v_n))) \\
          =  f(x_n,w_n) - g(H(f(x_n,w_n),v_n)) + g(\hat{y}_{n+1}) ,
\end{multline}
where  the first two equations have been inserted into the last.  This is the
filter equation for tracking and identifying the extended state of \refeq{eq:abs-RV}.
One may observe that the normal evolution model \refeq{eq:abs-RV} is corrected
by the innovation term.  This is the \emph{best unbiased} filter,
with $\phi(\hat{y})$ a MMSE estimate.
It is clear that the \emph{stability} of the solution to \refeq{eq:update-xa1-filt}
will depend on the contraction properties or otherwise of the map
$f - g \circ H \circ f = (I-g \circ H) \circ f$ as applied to $x_n$, but that is
not completely worked out yet and beyond the scope of this paper. 

By combining the minimisation property~\refeq{eq:min-ce} and the Doob-Dynkin 
lemma~\refeq{eq:DDL}, we see that the map $\phi_\Psi$ is defined by
\begin{equation} \label{eq:min-CE}
  \nd{\Psi(x) - \phi_\Psi(y)}^2_{\E{X}} = \min_{\vpi} \nd{\Psi(x) - \vpi(y)}^2_{\E{X}}
  = \min_{z\in\E{X}_{\infty}} \nd{\Psi(x) - z}^2_{\E{X}},
\end{equation}
where $\vpi$ ranges over all measurable maps $\vpi:\C{Y}\to\C{X}$.
As $\E{X}_{\sigma(y)}=\E{X}_{\infty}$ is $\C{L}$-closed \citep{Bosq2000, HgmEzBvrAlOp15},
it is characterised similarly to
\refeq{eq:var-ce}, but by orthogonality in the $\C{L}$-invariant sense
\begin{equation} \label{eq:var-CE}
  \forall z\in\E{X}_{\infty}:\quad \EXP{z \otimes (\Psi(x) - \phi_\Psi(y))} = 0,
\end{equation}
i.e.\ the RV $(\Psi(x) - \vpi(y))$ is orthogonal in the $\C{L}$-invariant sense
to all RVs $z\in\E{X}_{\infty}$, which means its correlation operator vanishes.
Although the CE
$\EXP{x|y} = P_{\sigma(y)}(x)$ is an orthogonal projection,
as the measurement operator $Y$, resp.\ $h$ or $H$, which evaluates $y$, is not
necessarily linear in $x$,  hence the optimal map $\phi_x(y)$ is also not
necessarily linear in $y$.  In some sense it has to be the opposite of $Y$.

\subsection{The linear filter}     \label{SS:linear-filter}
The minimisation in \refeq{eq:min-CE} over all measurable maps is still a
formidable task, and typically only feasible in an approximate way.  
One problem of course is, that the space $\E{X}_{\infty}$ is in general
infinite-dimensional.  The standard Galerkin approach is then to approximate
it by finite-dimensional subspaces, see \citep{HgmEzBvrAlOp15} for a
general description and analysis of the Galerkin convergence.

Here we replace  $\E{X}_{\infty}$ by much smaller subspace; and we choose
in some way the simplest possible one
\begin{equation} \label{eq:X-1}
\E{X}_1 = \{ z \;:\; z = \Phi(y) = L (y(\omega)) + b, \;
 L \in \E{L}(\C{Y},\C{X}),\; b \in \C{X}  \} \subset \E{X}_{\infty}
 \subset \E{X} ,
\end{equation}
where the $\Phi$ are just \emph{affine} maps; they are certainly measurable.  Note that
$\E{X}_1$ is also an $\C{L}$-invariant subspace of $\E{X}_{\infty}\subset\E{X}$.

Note that also other, possibly larger, $\C{L}$-invariant subspaces of $\E{X}_{\infty}$
can be used, but this seems to be smallest useful one.
Now the minimisation \refeq{eq:min-CE} may be replaced by 
\begin{equation} \label{eq:min-CE-lin}
  \nd{x - (K(y)+a)}^2_{\E{X}} =   \min_{L, b} \nd{x - (L(y)+b)}^2_{\E{X}} ,
\end{equation}
and the optimal affine map is defined by $K\in \E{L}(\C{Y},\C{X})$ and
$a \in \C{X}$.

Using this $g(y) = K(y) + a$, one disregards some
information as $\E{X}_1 \subset \E{X}_{\infty}$ is usually a true subspace --- 
observe that the subspace represents the information we may learn from
the measurement --- but the computation is easier, and one arrives in lieu of
\refeq{eq:update-xa1} at
\begin{equation} \label{eq:update-xa1L}
  x_{a,1L} =  x_f + ( K(\hat{y}) - K(y) )=    x_f + K(\hat{y} - y).
\end{equation}
This is the \emph{best linear} filter, with the linear MMSE $K(\hat{y})$.
One may note that the constant term
$a$ in \refeq{eq:min-CE-lin} drops out in the filter equation.

The algorithm corresponding to \refeq{eq:update-xa1-filt} is then
\begin{multline}      \label{eq:update-xa1L-filt}
  x_{n+1} =  f(x_n,w_n) + K((\hat{y}_{n+1}) - H(f(x_n,w_n),v_n))  \\
          =  f(x_n,w_n) - K(H(f(x_n,w_n),v_n)) + K(\hat{y}_{n+1}) .
\end{multline}

\subsection{The Gauss-Markov theorem and the Kalman filter}      \label{SS:GM-and-KF}
The optimisation described in \refeq{eq:min-CE-lin} is a familiar one, it is
easily solved, and the solution is given by an extension of the \emph{Gauss-Markov}
theorem \citep{Luenberger1969}.  The same idea of a linear MMSE is behind the
\emph{Kalman} filter 
\citep{Kalman, Grewal2008, Goldstein2007, Papoulis1998/107, Evensen2009}.
In our context it reads
\begin{thm}
The solution to \refeq{eq:min-CE-lin}, minimising
\[ \nd{x - (K(y)+a)}^2_{\E{X}} =  
  \min_{L, b} \nd{x - (L(y)+b)}^2_{\E{X}} \]
is given by $K := \cov(x,y) \cov(y)^{-1}$ and
$a := \bar{x} - K(\bar{y})$, where $\cov(x,y)$
is the covariance of $x$ and $y$, and $\cov(y)$ is the
auto-covariance of $y$.    In case $\cov(y)$ is \emph{singular}
or nearly singular, the \emph{pseudo-inverse} can be taken
instead of the inverse.
\end{thm}
The operator $K\in\E{L}(\C{Y},\C{X})$ is also called the \emph{Kalman} gain, and has
the familiar form known from least squares projections.  It is interesting
to note that initially the connection between MMSE and Bayesian estimation
was not seen \citep{McGrayne2011}, although it is one of the simplest approximations.

The resulting filter \refeq{eq:update-xa1L}
is therefore called the \tbf{Gauss-Markov-Kalman} filter (GMKF).
The original Kalman filter has \refeq{eq:update-xa1L} just for the means 
\[ \bar{x}_{a,1L} = \bar{x}_f + K(\hat{y} - \bar{y}) .\]
It easy to compute that
\begin{thm}
The covariance operator corresponding to \refeq{eq:cov-xa1}
$\mrm{cov}(x_{a,1L})$ of $x_{a,1L}$ is given by
\[ \mrm{cov}(x_{a,1L}) = \mrm{cov}(x_f) - K \mrm{cov}(x_f,y)^T
    = \mrm{cov}(x_f)-\mrm{cov}(x_f,y) \mrm{cov}(y)^{-1} \mrm{cov}(x_f,y)^T , \]
which is Kalman's formula for the covariance.
\end{thm}
This shows that \refeq{eq:update-xa1L} is a true extension of the classical Kalman filter (KF).
Rewriting \refeq{eq:update-xa1L} explicitly in less symbolic notation
\begin{equation} \label{eq:GMKF-RV}
  x_a(\omega) =  x_f(\omega) + 
      \cov(x_f,y)\cov(y)^{-1}(\hat{y} - y(\omega)) ,
\end{equation}
one may see that it is a relation between RVs, and hence some further \emph{stochastic}
discretisation is needed to be numerically implementable.

\section{Nonlinear filters} \label{S:nonlin-filt}
The derivation of nonlinear but polynomial filters is given
in  \citep{HgmEzBvrAlOp15}.  It has the advantage of showing the connection
to the ``Bayes linear'' approach \citep{Goldstein2007}, to the
Gauss-Markov theorem \citep{Luenberger1969}, and to the \emph{Kalman}
filter \citep{Kalman} \citep{Papoulis1998/107}.  Correcting higher moments of
the posterior RV has been touched on in \refSS{correct-higher}, and is not
the topic here.  Now the focus is on computing better than linear
(see \refSS{linear-filter}) approximations to the CE operator, which is
the basic tool for the whole updating and identification process.

We follow \citep{HgmEzBvrAlOp15} for a more general approach not limited to polynomials,
and assume a set of linearly independent measurable functions, not necessarily orthonormal, 
\begin{equation}  \label{eq:basis-L0}
  \C{B} := \{\psi_\alpha \; | \; \alpha \in \C{A},  \;
   \psi_\alpha(y(\omega)) \in \C{S}\}  \subseteq \C{S}_\infty   
\end{equation}
where $\C{A}$ is some countable index set.  Galerkin convergence \citep{HgmEzBvrAlOp15}
will require that
\[ \C{S}_\infty = \overline{\spn}\; \C{B}, \] 
i.e.\ that $\C{B}$ is a Hilbert basis of $\C{S}_\infty$.

Let us consider a general function $\Psi:\C{X}\to\C{R}$ of $x$, where
$\C{R}$ is some Hilbert space, of which we want to compute the conditional
expectation $\EXP{\Psi(x)|y}$.
Denote by $\C{A}_k$ a finite part of $\C{A}$ of cardinality $k$, such
that $\C{A}_k \subset \C{A}_\ell$ for $k<\ell$ and $\bigcup_k \C{A}_k
=\C{A}$, and set
\begin{equation}  \label{eq:def-Rn}
\E{R}_k :=  \C{R} \otimes \C{S}_k \subseteq \E{R}_\infty := \C{R} \otimes \C{S}_\infty,
\end{equation}
where the finite dimensional and hence closed subspaces $\C{S}_k$ are
given by
\begin{equation}  \label{eq:def-Sn}
  \C{S}_k := \spn \{\psi_\alpha \; | \; \alpha \in
      \C{A}_k, \; \psi_\alpha \in \C{B} \} \subseteq \C{S} .
\end{equation}
Observe that the spaces $\E{R}_k$ from \refeq{eq:def-Rn} are $\C{L}$-closed,
see \citep{HgmEzBvrAlOp15}.  In practice, also a ``spatial'' discretisation
of the spaces $\C{X}$ resp.\ $\C{R}$ has to be carried out; but this is
a standard process and will be neglected here for the sake of brevity and
clarity.

For a RV $\Psi(x) \in \E{R} = \C{R} \otimes \C{S}$ we make the following `ansatz' for the
optimal map $\phi_{\Psi,k}$ such that $P_{\E{R}_k}(\Psi(x)) =
\phi_{\Psi,k}(y)$:
\begin{equation}  \label{eq:ansatz-psi}
   \Phi_{\Psi,k}(y) = \sum_{\alpha \in \C{A}_k} v_\alpha \psi_\alpha(y),
\end{equation}
with as yet unknown coefficients $v_\alpha \in \C{R}$.  This is a
normal \emph{Galerkin}-ansatz, and the Galerkin orthogonality
\refeq{eq:var-CE} can be used to determine these coefficients.

Take $\C{Z}_k := \D{R}^{\C{A}_k}$ with canonical basis
$\{\vek{e}_\alpha \; | \; \alpha \in \C{A}_k \}$, and let 
\[ \vek{G}_k
:= (\ip{\psi_\alpha(y(x))}{\psi_\beta(y(x))}_{\C{S}})_{\alpha, \beta
  \in \C{A}_k} \in \E{L}(\C{Z}_k)
\]
be the symmetric positive definite Gram matrix of the basis of
$\C{S}_k$; also set
\begin{align*}
  \tnb{v} &:= \sum_{\alpha \in \C{A}_k} \vek{e}_\alpha \otimes v_\alpha
  \in \C{Z}_k \otimes \C{R}, \\
  \tnb{r} &:= \sum_{\alpha \in \C{A}_k} \vek{e}_\alpha \otimes
  \EXP{\psi_\alpha(y(x)) R(x)} \in \C{Z}_k \otimes \C{R}.
\end{align*}
\begin{thm}
  For any $k \in \D{N}$, the coefficients $\{ v_\alpha\}_{\alpha \in
    \C{A}_k}$ of the optimal map $\phi_{\Psi,k}$ in \refeq{eq:ansatz-psi}
  are given by the unique solution of the Galerkin equation
  \begin{equation} \label{eq:psi-Galerkin}
    (\vek{G}_k \otimes I_{\C{R}}) \tnb{v} = \tnb{r} .
  \end{equation}
It has the formal solution 
\[ \tnb{v} = (\vek{G}_k \otimes I_{\C{R}})^{-1} \tnb{r} = 
    (\vek{G}_k^{-1} \otimes I_{\C{R}}) \tnb{r} \in \C{Z}_k \otimes \C{R} . \]
\end{thm}
\begin{proof}
  The Galerkin \refeq{eq:psi-Galerkin} is a simple consequence of
  the Galerkin orthogonality \refeq{eq:var-CE}.  As the Gram
  matrix $\vek{G}_k$ and the identity $I_{\C{R}}$ on $\C{R}$ are
  positive definite, so is the tensor operator $(\vek{G}_k \otimes
  I_{\C{R}})$, with inverse $(\vek{G}_k^{-1} \otimes I_{\C{R}})$.
\end{proof}
The block structure of the equations is
clearly visible.  Hence, to solve \refeq{eq:psi-Galerkin}, one only has
to deal with the `small' matrix $\vek{G}_k$.

The update corresponding to \refeq{eq:update-xa1-filt},
using again $\Psi(x) = x$, one obtains
a possibly nonlinear filter based on the basis $\C{B}$:
\begin{equation} \label{eq:GBU-xg}
  x_a \approx x_{a,k} =  x_f + \left( \phi_{x,k}(\hat{y}) - 
  \phi_{x,k}(y(x_f)) \right)  = x_f +  x_{\infty,k}.
\end{equation}
In case that $\C{Y}^* \subseteq \spn \{ \psi_\alpha\}_{\alpha \in
  \C{A}_k}$, i.e.\ the functions with indices in $\C{A}_k$ generate
all the linear functions on
$\C{Y}$, this is a true extension of the Kalman filter.

Observe that this allows one to compute the map in \refeq{eq:CE-DDL} or rather
\refeq{eq:Post-Psi} to any desired accuracy.  Then, using this tool, one may
construct a new random variable which has the desired posterior expectations;
as was started in \refSS{correct-mean} and \refSS{correct-higher}.  This is
then a truly nonlinear extension of the linear filters described in \refS{GMKF},
and one may expect better tracking properties than even for the best linear filters.
This could for example allow for less frequent observations of a dynamical system.

\section{Numerical realisation}   \label{S:num-realisation}
This is only going to be a rough overview on possibilities of numerical realisations.
Only the simplest case of the linear filter will be considered, all other approximations
can be dealt with in an analogous manner.
Essentially we will look at two different kind of approximations, \emph{sampling} and
\emph{functional} or \emph{spectral} approximations.

\subsection{Sampling}        \label{SS:sampling}
As starting point take \refeq{eq:GMKF-RV}.   As it is a relation between RVs, it
certainly also holds for \emph{samples} of the RVs.  Thus it is possible to
take an \emph{ensemble} of sampling points $\omega_1,\dots,\omega_N$ and
require
\begin{equation} \label{eq:EnKF}  \forall \ell = 1,\dots,N:\quad
  \vek{x}_a(\omega_\ell) =  \vek{x}_f(\omega_\ell) + 
      \vek{C}_{x_f y} \vek{C}^{-1}_{y}(\hat{y} - y(\omega_\ell)) ,
\end{equation}
and this is the basis of the \emph{ensemble} Kalman filter, the EnKF \citep{Evensen2009};
the points $\vek{x}_f(\omega_\ell)$ and $\vek{x}_a(\omega_\ell)$ are sometimes
also denoted as \emph{particles}, and \refeq{eq:EnKF} is a simple version of
a \emph{particle filter}.  In \refeq{eq:EnKF}, $\vek{C}_{x_f y}=\cov(x_f,y)$ and
$\vek{C}_{y}=\cov(y)$

Some of the main work for the EnKF consists in obtaining
good estimates of $\vek{C}_{x_f y}$ and $\vek{C}_{y}$, as they have to be computed
from the samples.  Further approximations are possible, for example such as
\emph{assuming} a particular form for $\vek{C}_{x_f y}$ and $\vek{C}_{y}$.
This is the basis for methods like \emph{kriging} and \emph{3DVAR} resp.\ \emph{4DVAR},
where one works with an approximate Kalman gain $\vtil{K} \approx \vek{K}$.
For a recent account see \citep{KellyLawStuart-2014}.

\subsection{Functional approximation}        \label{SS:func-approx}
Here we want to pursue a different tack, and want to discretise RVs not through
their samples, but by \emph{functional} resp.\ \emph{spectral approximations} 
\citep{matthies6, xiu2010numerical, LeMatre10}.  This means that all RVs,
say $\vek{v}(\omega)$, are described as functions of \emph{known} RVs
$\{\xi_1(\omega),\dots,\xi_\ell(\omega),\dots\}$.  Often, when for example
stochastic processes or random fields are involved,  one has to deal here
with \emph{infinitely} many RVs, which for an actual computation have to be
truncated to a finte vector $\vek{\xi}(\omega)=[\xi_1(\omega),\dots,\xi_n(\omega)]$
of significant RVs.  We shall assume that these have been chosen such as to be
independent.  As we want to approximate later $\vek{x}=[x_1,\dots,x_n]$, we
do not need more than $n$ RVs $\vek{\xi}$.

One further chooses a finite set of linearly independent functions 
$\{\psi_\alpha\}_{\alpha\in\C{J}_M}$ of the variables $\vek{\xi}(\omega)$, where
the index $\alpha$ often is a \emph{multi-index}, and the set $\C{J}_M$ is a finite
set with cardinality (size) $M$.  Many different systems of functions can be
used, classical choices are \citep{matthies6,  xiu2010numerical, LeMatre10} multivariate
polynomials --- leading to the \emph{polynomial chaos expansion} (PCE), as well as
trigonometric functions, kernel functions as in kriging, radial basis functions,
sigmoidal functions as in artificial neural networks (ANNs), or functions derived
from fuzzy sets.  The particular choice is immaterial for the further development.
But to obtain results which match the above theory as regards $\C{L}$-invariant
subspaces, we shall assume that the set $\{\psi_\alpha\}_{\alpha\in\C{J}_M}$ includes
all the \emph{linear} functions of $\vek{\xi}$.  This is easy to achieve with
polynomials, and w.r.t\ kriging it corresponds to \emph{universal} kriging.
All other functions systems can also be augmented by a linear trend.

Thus a RV $\vek{v}(\omega)$ would be replaced by a functional approximation
\begin{equation}  \label{eq:FA}
   \vek{v}(\omega) = \sum_{\alpha\in\C{J}_M} \vek{v}_\alpha 
     \psi_\alpha(\vek{\xi}(\omega)) = \sum_{\alpha\in\C{J}_M} \vek{v}_\alpha 
     \psi_\alpha(\vek{\xi}) = \vek{v}(\vek{\xi}) .
\end{equation}
The argument $\omega$ will be omitted from here on, as we transport the probability
measure $\D{P}$ on $\Omega$ to $\vek{\Xi}=\Xi_1 \times \dots \times \Xi_n$, the range
of $\vek{\xi}$, giving $\D{P}_{\xi} = \D{P}_1 \times \dots \times \D{P}_n$ as
a product measure,
where $\D{P}_\ell = (\xi_\ell)_* \D{P}$ is the distribution measure of the
RV $\xi_\ell$, as the RVs $\xi_\ell$ are independent.  All computations
from here on are performed on $\vek{\Xi}$, typically some subset of $\D{R}^n$.
Hence $n$ is the dimension of our problem, and if $n$ is large, one faces a
high-dimensional problem.  It is here that low-rank tensor approximations
\citep{Hackbusch_tensor} become practically important.

It is not too difficult to see that the linear filter when applied to the spectral
approximation has exactly the same form as shown in \refeq{eq:GMKF-RV}.  Hence the
basic formula \refeq{eq:GMKF-RV} looks formally the same in both cases, once it is
applied to samples or ``particles'', in the other case to the functional approximation
of RVs, i.e.\ to the coefficients in \refeq{eq:FA}.

In both of the cases described here in \refSS{sampling} and in this \refSS{func-approx},
the question as how to compute the covariance matrices in \refeq{eq:GMKF-RV} arises.
In the EnKF in \refSS{sampling} they have to be computed from the samples \citep{Evensen2009},
and in the case of functional resp.\ spectral approximations they can be computed
from the coefficients in \refeq{eq:FA}, see \citep{opBvrAlHgm12, BvrAkJsOpHgm11}.

In the sampling context, the samples or particles may be seen as $\updelta$-measures,
and one generally obtains weak-$*$ convergence of convex combinations of these
$\updelta$-measures to the continuous limit as the number of particles increases.
In the case of functional resp.\ spectral approximation
one can bring the whole theory of Galerkin-approximations to bear on the problem,
and one may obtain convergence of the involved RVs in appropriate norms \citep{HgmEzBvrAlOp15}.
We leave this topic with this pointer to the literature, as this is too extensive
to be discussed any further and hence is beyond the scope of the present work.

\section{Examples}   \label{S:examples}
\begin{figure}[!ht]
\centering
 \includegraphics[width=0.9\textwidth,height=0.35\textheight]{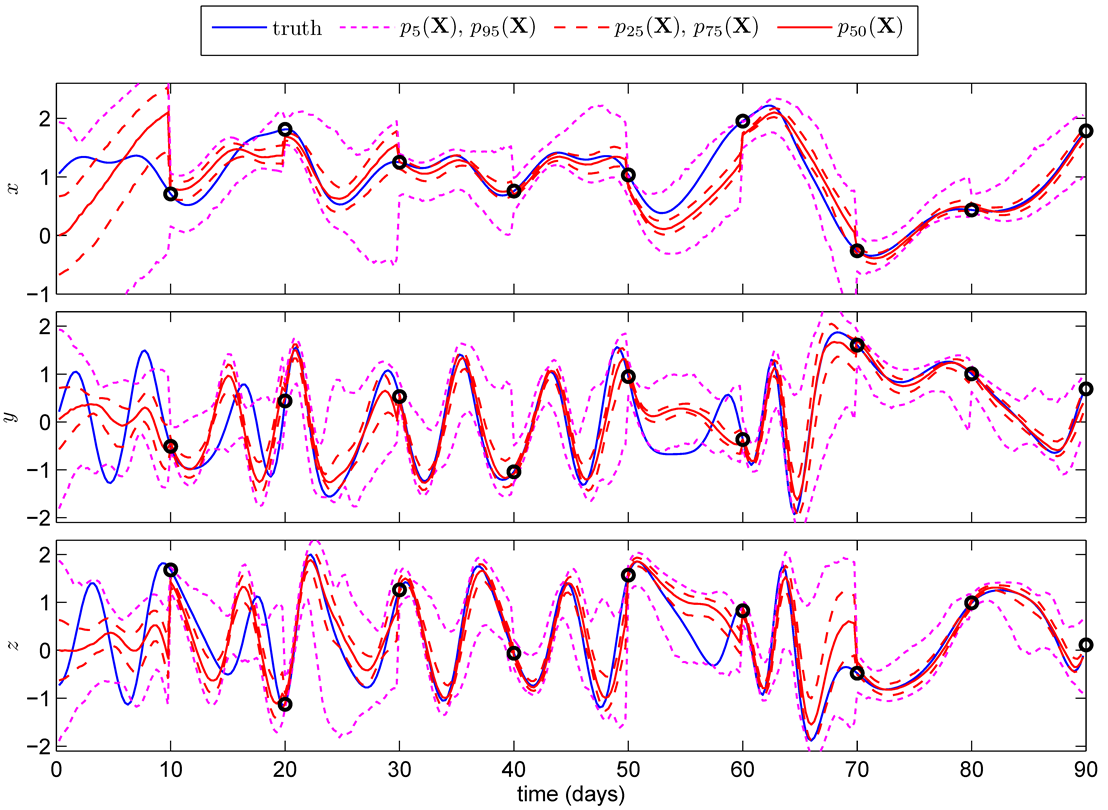}
 \caption[Identification of Lorenz-84 model]{Time evolution of the Lorenz-84 model with
 state identification with the LBU, from \citep{opBvrAlHgm12}.  For the estimated state
 uncertainty the $50\%$ (full line), $\pm 25\%$, and $\pm 45\%$ quantiles are shown.}
\label{F:exp-B-1}
\end{figure}
The first example is a dynamic system considered in \citep{opBvrAlHgm12},
it is the well-known Lorenz-84 chaotic model, a system of three nonlinear
ordinary differential equations operating in the chaotic regime. This
is an example along the description of \refeq{eq:abstr-sys-I} and
\refeq{eq:obs-I} in \refSS{data-m}.  Remember that this was
originally a model to describe the evolution of some amplitudes of a
spherical harmonic expansion of variables describing world climate.
As the original scaling of the variables has been kept, the time axis
in \refig{exp-B-1} is in \emph{days}.  Every ten days a noisy
measurement is performed and the state description is updated.  In
between the state description evolves according to the chaotic dynamic
of the system.  One may observe from \refig{exp-B-1} how the
uncertainty---the width of the distribution as given by the quantile
lines---shrinks every time a measurement is performed, and then
increases again due to the chaotic and hence noisy dynamics.  Of
course, we did not really measure world climate, but rather simulated
the ``truth'' as well, i.e.\ a \emph{virtual} experiment, like the
others to follow.  More details may be found in \citep{opBvrAlHgm12}
and the references therein.  All computations are performed in a
functional approximation with  polynomial chaos expansions as
alluded to in \refSS{func-approx}, and the filter is linear according to \refeq{eq:GMKF-RV}.

To introduce the nonlinear filter as sketched in \refS{nonlin-filt},
where for the nonlinear filter the functions in \refeq{eq:ansatz-psi} included
polynomials up to quadratic terms, one may
look shortly at a very simplified example, identifying a value, where
only the third power of the value plus a Gaussian error RV is observed.
All updates follow \refeq{eq:update-xa1}, but the update map is computed
with different accuracy.
\begin{figure}[!ht]
\centering
 \includegraphics[width=0.6\textwidth]{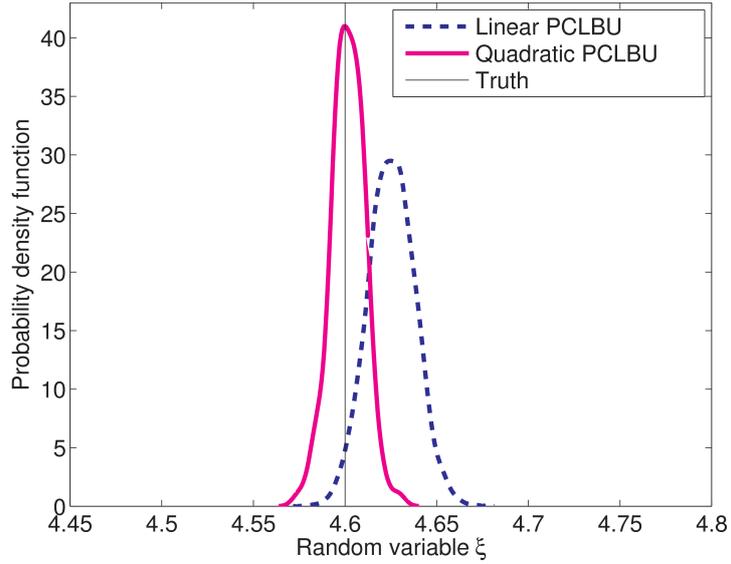}
 \caption[Updates for cubic observations]{Perturbed observations of the cube of a RV, different
 updates --- linear, and quadratic update}
\label{F:exp-B-2}
\end{figure}
Shown are the pdfs produced by the linear filter according to \refeq{eq:GMKF-RV}
--- Linear polynomial chaos Bayesian update (Linear PCBU) ---
a special form of \refeq{eq:update-xa1},
and using polynomials up to order two, the quadratic polynomial chaos
Bayesian update (QPCBU).  One may observe that due to the nonlinear observation,
the differences between the linear filters and the quadratic one are already significant,
the QPCBU yielding a better update.

We go back to the example shown in \refig{exp-B-1}, but now consider
only for one step a nonlinear filter like in \refig{exp-B-2},
see \citep{HgmEzBvrAlOp15}.
\begin{figure}[!ht]
\centering
 \includegraphics[width=0.9\textwidth,height=0.2\textheight]{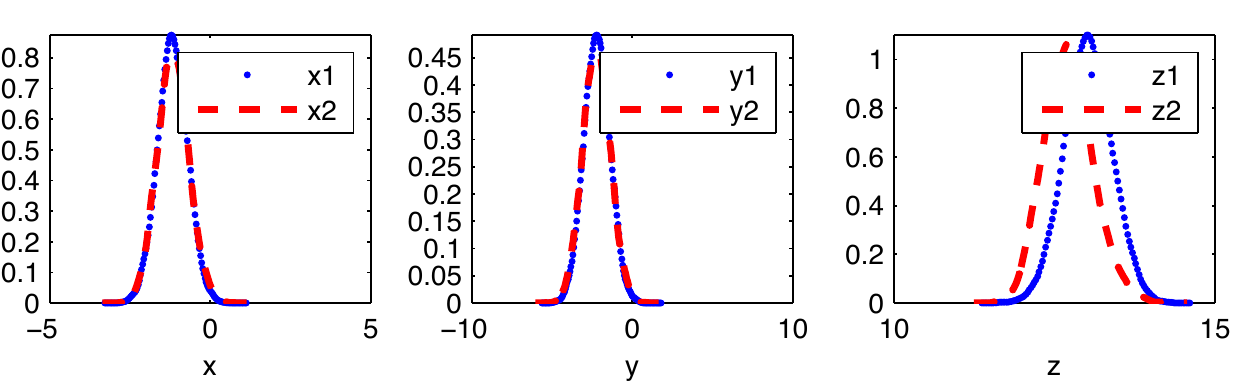}
 \caption[Lorenz-84, linear observations with different updates]{Lorenz-84 model, perturbed
 linear observations of the state: Posterior for LBU and 
 QBU after one update, from \citep{HgmEzBvrAlOp15}}
\label{F:exp-B-3}
\end{figure}
As a first set of experiments we take the measurement operator to be
linear in the state variable to be identified, i.e.\ we can observe the \emph{whole} state
directly.  At the moment we consider updates after each day---whereas
in \refig{exp-B-1} the updates were performed every 10 days.  The
update is done once with the linear Bayesian update (LBU), and again
with a \emph{quadratic} nonlinear BU (QBU).  The results
for the posterior pdfs are given in \refig{exp-B-3}, where the
linear update is dotted in blue and labelled $z1$, and the full red line is the
quadratic QBU labelled $z2$; there is hardly any difference between the two
except for the $z$-component of the state,  most
probably indicating that the LBU is already very accurate.

Now the same experiment, but the measurement operator is cubic:
\begin{figure}[!ht]
\centering
 \includegraphics[width=0.9\textwidth,height=0.2\textheight]{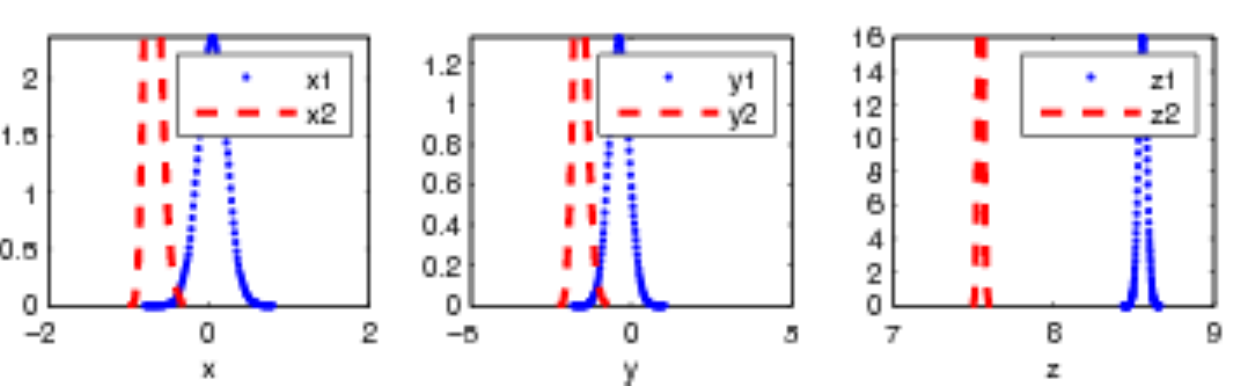}
 \caption[Lorenz-84, cubic observations with different updates]{Lorenz-84 model, perturbed
 cubic observations of the state: Posterior for LBU and 
 QBU after one update, from \citep{HgmEzBvrAlOp15}}
\label{F:exp-B-4}
\end{figure}
These differences in posterior pdfs after one update may be gleaned 
from \refig{exp-B-4}, and they are indeed larger
than in the linear case \refig{exp-B-3}, due to the strongly nonlinear
measurement operator, showing that the QBU may provide much more
accurate tracking of the state, especially for non-linear observation
operators.

\begin{figure}[!ht]
\centering
\begin{minipage}{0.47\textwidth}
  \centering
  \includegraphics[width=0.99\linewidth,height=0.26\textheight]{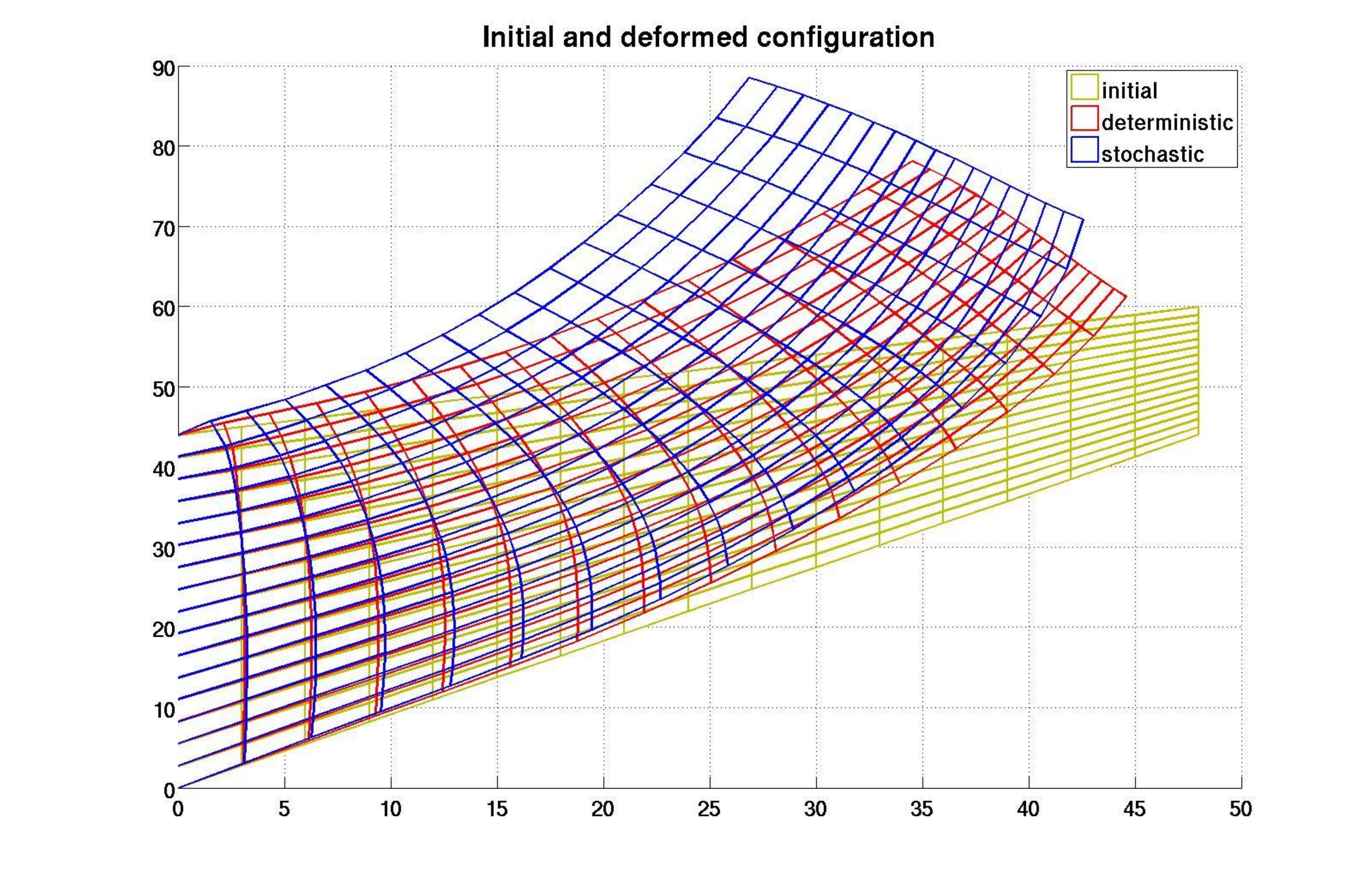}
  \caption[Cook's membrane --- large strain elasto-plasticity]{Cook's membrane --- large
  strain elasto-plasticity, undeformed grid [initial], deformations with mean properties
  [deterministic], and mean of the deformation with stochastic properties [stochastic], from \citep{BvrAkJsOpHgm11}, \citep{rosic2013hgm}, \citep{HgmEzBvrAlOp15}}
  \label{F:exp-B-5}
\end{minipage}%
\hfill
\begin{minipage}{0.47\textwidth}
  \centering
  \includegraphics[width=0.99\linewidth]{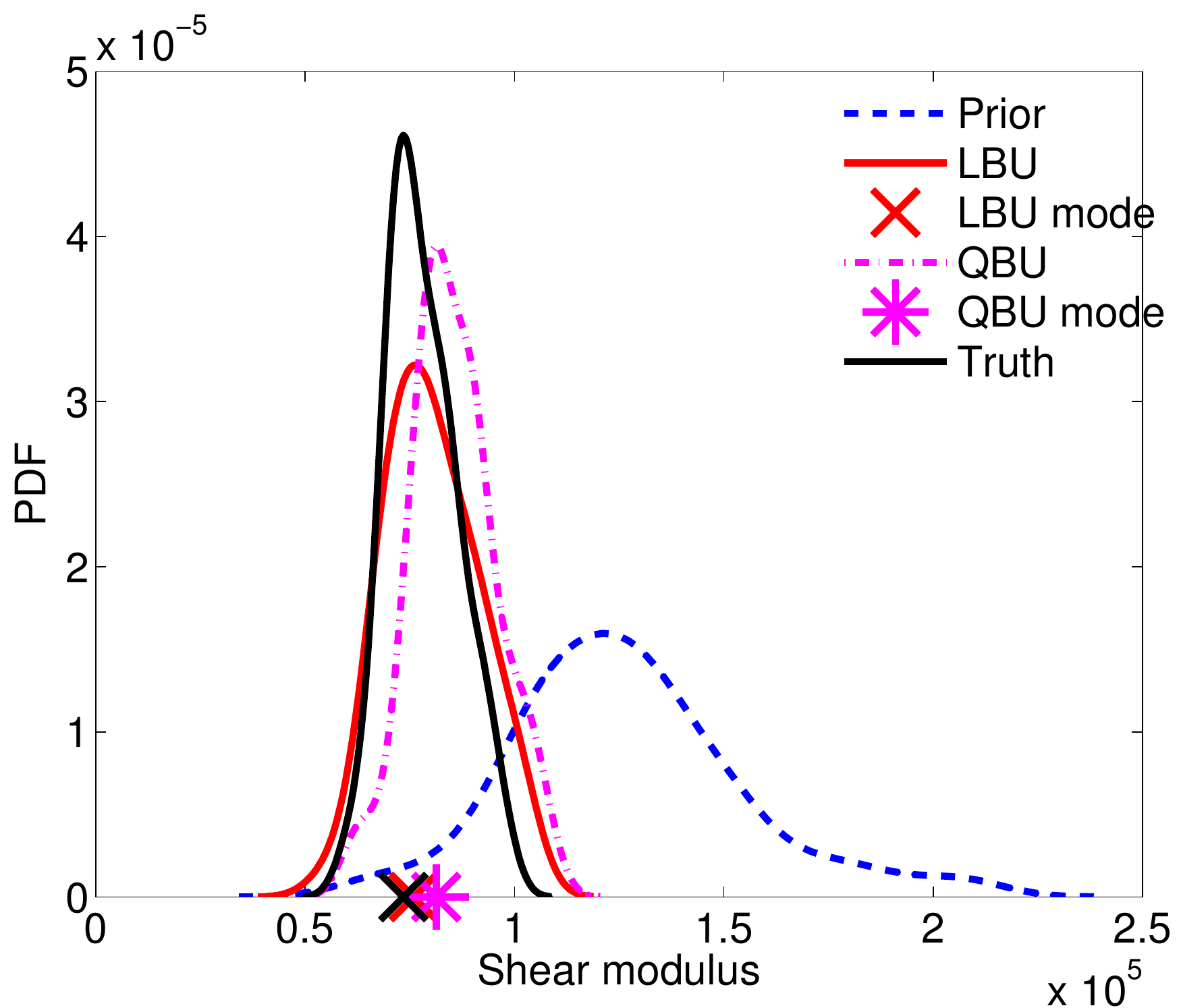}
  \caption[Cook's membrane --- Update of shear modulus]{Cook's membrane --- large
  strain elasto-plasticity, perturbed linear observations of the deformation,
  LBU and QBU for the shear modulus, from \citep{HgmEzBvrAlOp15}}
  \label{F:exp-B-6}
\end{minipage}
\end{figure}
As a last example we follow \citep{HgmEzBvrAlOp15} and take a strongly nonlinear and also
non-smooth situation, namely elasto-plasticity with linear hardening
and large deformations and a \emph{Kirchhoff-St.\ Venant} elastic
material law \citep{BvrAkJsOpHgm11}, \citep{rosic2013hgm}.  This
example is known as \emph{Cook's membrane}, and is shown in
\refig{exp-B-5} with the undeformed mesh (initial), the deformed one
obtained by computing with average values of the elasticity and
plasticity material constants (deterministic), and finally the average
result from a stochastic forward calculation of the probabilistic
model (stochastic), which is described by a variational inequality
\citep{rosic2013hgm}.

The shear modulus $G$, a random field and not a deterministic value in
this case, has to be identified, which is made more difficult by the
non-smooth non-linearity.  In \refig{exp-B-6} one may see the `true'
distribution at one point in the domain in an unbroken black line,
with the mode --- the maximum of the pdf --- marked by a black cross
on the abscissa, whereas the prior is shown in a dotted blue line.
The pdf of the LBU is shown in an unbroken red line, with its mode
marked by a red cross, and the pdf of the QBU is shown in a broken
purple line with its mode marked by an asterisk.  Again we see a
difference between the LBU and the QBU.  But here a curious thing
happens; the mode of the LBU-posterior is actually closer to the mode
of the `truth' than the mode of the QBU-posterior.  This means that
somehow the QBU takes the prior more into account than the LBU, which
is a kind of overshooting which has been observed at other occasions.
On the other hand the pdf of the QBU is narrower --- has less
uncertainty --- than the pdf of the LBU.

\section{Conclusion}   \label{S:conclusion}
A general approach for state and parameter estimation has been presented in a Bayesian framework.
The Bayesian approach is based here on the conditional expectation (CE) operator,
and different approximations were discussed, where the linear approximation leads to a
generalisation of the well-known Kalman filter (KF), and is here termed the Gauss-Markov-Kalman
filter (GMKF), as it is based on the classical Gauss-Markov theorem.
Based on the CE operator, various approximations
to construct a RV with the proper posterior distribution were shown, where just correcting
for the mean is certainly the simplest type of filter, and also the basis of the GMKF.

Actual numerical computations typically require a discretisation of both the spatial variables
--- something which is practically independent of the considerations here --- and the
stochastic variables.  Classical are sampling methods, but here the use of spectral resp.\
functional approximations is alluded to, and all computations in the examples shown
are carried out with functional approximations.


\providecommand{\bysame}{\leavevmode\hbox to3em{\hrulefill}\thinspace}
\providecommand{\MR}{\relax\ifhmode\unskip\space\fi MR }
\providecommand{\MRhref}[2]{%
  \href{http://www.ams.org/mathscinet-getitem?mr=#1}{#2}
}
\providecommand{\href}[2]{#2}



%

\end{document}